\documentclass[11pt,a4paper,oneside]{amsart}

\usepackage{geometry, amsmath, amssymb, amsthm,amsfonts,mathabx,graphicx,mathtools,nicematrix,microtype,tikz,float,url,cite}

\usepackage{xcolor}
\definecolor{linkblue}{RGB}{1,1,190}
\definecolor{citered}{RGB}{190,1,1}
\definecolor{recursionblue}{RGB}{70,130,180}

\usepackage[linkcolor=linkblue
,urlcolor=linkblue
,citecolor=citered
,colorlinks
,bookmarksopen=true,
]{hyperref}

\usepackage[T1]{fontenc}
\usepackage[utf8]{inputenc}

\makeatletter
\AtBeginDocument{
	\hypersetup{
		pdftitle  = {\@title},
	}
}
\makeatother

\usetikzlibrary{positioning}

\newtheorem{thm}{Theorem}[section]

\newtheorem{lem}[thm]{Lemma}
\newtheorem{cor}[thm]{Corollary}
\theoremstyle{definition}
\newtheorem{defn}[thm]{Definition}
\newtheorem{notn}[thm]{Notation}
\newtheorem{exa}[thm]{Example}
\newtheorem{rem}[thm]{Remark}

\newcommand{\gzero}{0}
\newcommand{\M}{\mathrm{M}}
\newcommand{\T}{\mathrm{T}}

\newcommand{\MnD}{\M_n(D)}
\newcommand{\MnK}{\M_n(K)}
\newcommand{\MnR}{\M_n(R)}

\newcommand{\TnD}{\T_n(D)}
\newcommand{\TnK}{\T_n(K)}

\newcommand{\sLnD}{_s{\mathrm L}_n(D)}
\newcommand{\sLnR}{_s{\mathrm L}_n(R)}

\newcommand{\Mprec}{\M_{\precsim}}
\newcommand{\MprecD}{\Mprec(D)}
\newcommand{\MprecK}{\Mprec(K)}
\newcommand{\MprecR}{\Mprec(R)}
\newcommand{\MprecS}{\Mprec(S)}
\newcommand{\MprecI}{\Mprec(I)}

\DeclareMathOperator{\diag}{diag}
\DeclareMathOperator{\Int}{Int}
\DeclareMathOperator{\N}{N}

\newcommand{\ZZ}{\mathbb{Z}}

\title{Integer-valued polynomials over structural matrix rings}
\author{Valentin Havlovec}
\address{Institute of Analysis and Number Theory, Graz University of Technology,
	Kopernikusgasse 24, 8010 Graz, Austria}
\email{havlovec@math.tugraz.at}

\thanks{This research was funded in whole or in part by the Austrian Science Fund (FWF) [10.55776/P35788].}


\begin{document}
	
	\begin{abstract}
		We study integer-valued polynomials and null polynomials over structural
		matrix rings, that is, rings whose elements are matrices in which an entry may be
		nonzero only when its row index precedes its column index in a fixed
		preorder. We prove that the integer-valued polynomials over a structural matrix ring with entries in an integral domain form a ring, and that the null polynomials over a structural matrix ring with entries in an arbitrary commutative ring form a two-sided ideal. In both settings, we give a characterization of the corresponding polynomials with matrix coefficients in terms of scalar-coefficient polynomials.
		These results extend corresponding theorems for full matrix rings and
		upper triangular matrix rings.
	\end{abstract}
	\maketitle
	
	\section{Introduction}
	Over the past two decades, there have been numerous articles studying generalizations of integer-valued polynomials to noncommutative rings~\cite{Frisch_IVPonAlgebras,Frisch_PolFunOnUpperTriangularMatrixAlgebras,PeruginelliWerner_DecompIVP,SedighiHafshejaniNaghipourRismanchian_IVPoverBlockMatrixAlgebras,SedighiHafshejani_IVPSubsetsMatrices,Werner_IVPoverMatrixRings,Werner_IVPQuaternions,Werner_IVPSubsetsQuaternions}.
	For an overview, see the survey article~\cite{Werner_IVPAlgebrasSurvey}. The recent monograph by Chabert \cite{Chabert_IVP2025} covers many more aspects of integer-valued polynomials.
	A common framework in the noncommutative setting is the following~\cite{Frisch_IVPonAlgebras}: let $D$ be an integral domain with quotient field $K$, let $A$ be a module-finite, torsion-free $D$-algebra such that $A \cap K = D$ (where $A$ and $K$ are canonically embedded in $A\otimes_D K$). Set $B=A \otimes_D{K}$ and consider the set of integer-valued polynomials
	\[
	\Int_{B}(A) = \{f \in B[x] \mid \forall a\in A:\ f(a)\in A\}.
	\]
	Here $f(a)$ is defined as right substitution, i.e., if $f=\sum_{k}f_k x^k$, then $f(a) = \sum_{k} f_ka^k$.
	
	A central question in this approach is to determine whether the set $\Int_B(A)$ is closed under multiplication and thus forms a subring of $B[x]$. This is not a given, since for polynomials $f,g$ with noncommuting coefficients, we generally have $(fg)(a) \neq f(a)g(a)$.
	That the integer-valued polynomials form a ring is known for quaternions~\cite{Werner_IVPQuaternions}, full matrix algebras $\MnD$~\cite{Werner_IVPoverMatrixRings}, upper triangular matrix algebras $\TnD$~\cite{Frisch_PolFunOnUpperTriangularMatrixAlgebras}, and certain block matrix algebras $\sLnD$~\cite{SedighiHafshejaniNaghipourRismanchian_IVPoverBlockMatrixAlgebras}.
	Furthermore, until recently no algebra $A$ was known for which $\Int_B(A)$ fails to be a ring.
	However, a finite ring whose null polynomials do not form a two-sided ideal was constructed in~\cite{Havlovec_NullPolynomials}. This finite ring is the reduction modulo $2$ of a module-finite, free $\ZZ$-algebra. Together with the correspondence between null polynomials and integer-valued ones 
	\cite[Theorem 2.4]{Werner_PolynomialsThatKill}, this yields an algebra $A$ for which $\Int_B(A)$ is not a ring.
	
	Among the positive examples above, the matrix algebras $\MnD$, $\TnD$, and $\sLnD$ share a common form. In each case the algebra $A$ is obtained by choosing a subset $\Sigma$ of the set of position indices $\{1,\ldots,n\} \times \{1,\ldots,n\}$, and taking $A$ to be the set of matrices with arbitrary entries in positions $(i,j)$ satisfying $(i,j) \in \Sigma$ and zeros in all other positions.
	Of course, whether the set $A$ forms a unital subring of $\MnD$ depends on the choice of $\Sigma$. This occurs precisely when $\Sigma$, viewed as a binary relation on $\{1,\ldots,n\}$, is reflexive and transitive, i.e., a preorder.  Rings constructed in this way are called \emph{structural matrix rings}~\cite{SmithVanWyk_StructuralMatrixRings}. For a preorder $\precsim$, we denote the corresponding structural matrix ring over $D$ by $\MprecD$.
	
	We prove that for any preorder $\precsim$, the set $\Int_{\MprecK}(\MprecD)$ of integer-valued polynomials is a ring, thus generalizing~\cite[Theorem 2.6]{SedighiHafshejaniNaghipourRismanchian_IVPoverBlockMatrixAlgebras}, which in turn was a generalization of~\cite[Theorem 5.4]{Frisch_PolFunOnUpperTriangularMatrixAlgebras} and~\cite[Theorem 1.2]{Werner_IVPoverMatrixRings}. We further give a description of the elements of $\Int_{\MprecK}(\MprecD)$ in terms of matrices with scalar coefficients. This characterization specializes to the known description of $\Int_{\MnK}(\MnD)$ in~\cite[Theorem 7.2]{Frisch_IVPonAlgebras} and of $\Int_{\TnK}(\TnD)$ in~\cite[Corollary 5.3]{Frisch_PolFunOnUpperTriangularMatrixAlgebras}.
	
	As alluded to before, there is a related question concerning null polynomials: for a possibly noncommutative ring $T$, let $\N(T) = \{f \in T[x] \mid \forall a \in T: f(a)=0\}$ be the set of null polynomials over $T$. Is $\N(T)$ a two-sided ideal of $T[x]$? Again, due to the absence of a substitution homomorphism, this is not clear.
	In 2014, Werner first raised this question for finite rings~\cite{Werner_PolynomialsThatKill}. He also explained the connection to integer-valued polynomials for module-finite $\mathbb{Z}$-algebras $A$, namely that $\Int_{B}(A)$ is a ring if and only if $\N(T)$ is a two-sided ideal for all rings $T$ of the form $A/dA$ with $d\in \mathbb{Z} \setminus \{0\}$. Frisch subsequently extended this connection to the setting of algebras over arbitrary domains, and also in the direction of integer-valued polynomials over subsets (i.e., polynomials mapping a proper subset of $A$ back into $A$) \cite{Frisch_Ringsets}.
	
	Werner answered the question in the affirmative for large classes of finite rings, including local rings, semisimple rings, and matrix rings over commutative rings, leading him to conjecture that $\N(T)$ is always a two-sided ideal. Further work by Frisch in 2017 showed that the conjecture also holds for upper triangular matrix rings~\cite[Theorem 5.2]{Frisch_PolFunOnUpperTriangularMatrixAlgebras}. No counterexample was known until recently~\cite{Havlovec_NullPolynomials}. We show that for every structural matrix ring $\MprecR$ the set of null polynomials $\N(\MprecR)$ is a two-sided ideal, and give a description of its elements in terms of matrices with scalar coefficients.
	\section{Structural matrix rings}
	Throughout, all rings are unital, with subrings sharing the same identity, and all polynomial rings are taken in a central indeterminate.
	
	Let $R$ be a commutative ring, and denote by $\MnR$ the ring of $n \times n$ matrices with entries in $R$, where $n$ is some positive integer.
	
	Let $\precsim$ be a preorder on the set $\{1,\ldots,n\}$ of row and column indices.
	Define the subset $\MprecR$ of $\MnR$ to be the set of all matrices having arbitrary entries in those positions $(i,j)$ for which $i \precsim j$, and zeros in all other positions. More precisely,
	\[
	\MprecR = \left\{(a_{ij}) \in \MnR \mid a_{ij} \neq 0 \Rightarrow i \precsim j\right\}.
	\]
	
	As matrix addition is defined pointwise, this set is clearly closed under addition and it also contains the zero matrix.
	Because $\precsim$ is reflexive, it also contains the identity matrix. Moreover, the set $\MprecR$ is closed under matrix multiplication. To see this, let $(a_{ij}), (b_{kl}) \in \MprecR$. The entry in the $(i,l)$-th position of $(a_{ij}) \cdot (b_{kl})$ is the sum $\sum_{j} a_{ij}b_{jl}$. This sum is nonzero only if there exists some $j$ with $i \precsim j$ and $j \precsim l$. From the transitivity of $\precsim$, it follows that $i\precsim l$ and thus $(a_{ij})\cdot (b_{kl}) \in \MprecR$. We conclude that $\MprecR$ is a subring of $\MnR$, called a \emph{structural matrix ring}.

	\begin{rem}
		If $\precsim$ is the usual order $\leq$, then $\MprecR$ is the ring of upper triangular matrices over $R$.
		If $\precsim$ is the universal relation with $i \precsim j$ for all $i,j$, then $\MprecR$ is the full ring of $n \times n$ matrices.
		If $s$ is an integer with $1 \leq s \leq n$, and $\precsim$ is defined by $i \precsim j \Leftrightarrow i \leq j \text{ or } j>s$, then $\MprecR$ is the ring $\sLnR$ of~\cite{SedighiHafshejaniNaghipourRismanchian_IVPoverBlockMatrixAlgebras}.		
	\end{rem}
	
	\begin{notn}
		The preorder $\precsim$ induces an equivalence relation $\sim$ defined by $i \sim j$ if and only if $i \precsim j$ and $j \precsim i$. We write $[i]$ for the equivalence class of $i$.
		For two integers $i,j \in \{1,\ldots,n\}$, we define the two sets $[i,j] = \{ h \in \{1,\ldots,n\} \mid i \precsim h \text{ and } h \precsim j \}$ and $[i,\infty) = \{ h \in \{1,\ldots,n\} \mid i \precsim h \}$. Note that the set $[i,i]$ is exactly the equivalence class $[i]$.
	\end{notn}
	
	After relabeling the rows and columns if necessary, we can assume that the equivalence classes under $\sim$ are intervals. Thus $\MprecR$ consists of block matrices whose nonzero entries lie either in square blocks on the diagonal (with sizes matching the sizes of the equivalence classes under $\sim$) or in rectangular blocks outside the diagonal.
	\begin{exa}	
		Figure~\ref{fig:PreorderExample} describes a preorder $\precsim$ on $\{1,\ldots,15\}$ and its corresponding structural matrix ring. The diagram on the left is the Hasse diagram of the partial order induced by \(\precsim\) on the quotient set
		\(\{1,\ldots,15\}/\mathord{\sim}\). Thus, its vertices are the \(\sim\)-equivalence classes, and 
		the class containing $i$ is equal to or lies below the class containing $j$ if and only if $i\precsim j$. Note that the isolated class $\{10,11\}$ is not comparable to any other.
		
		On the right, we display the form
		of a general element of the structural matrix ring
		\(\MprecR\). In positions marked with $*$ there can be arbitrary entries from $R$, whereas the entries in the other positions must be zero.
		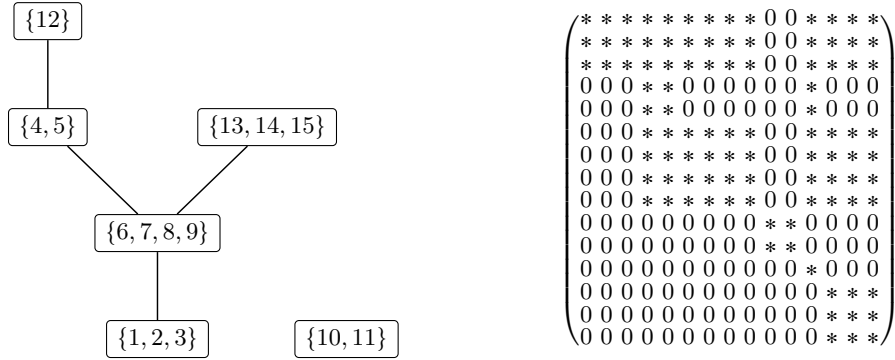
\begin{figure}[H]
			\begin{tikzpicture}[
				class/.style={
					draw,
					rounded corners=1.5pt,
					inner xsep=4pt,
					inner ysep=2.5pt,
					font=\footnotesize
				},
				cover/.style={
					line width=.55pt
				}
				]
				
				\node[class] (A) at (0,-2.1)
				{$\{1,2,3\}$};
				
				\node[class] (C) at (0,-.7)
				{$\{6,7,8,9\}$};
				
				\node[class] (B) at (-1.45,.7)
				{$\{4,5\}$};
				
				\node[class] (F) at (1.45,.7)
				{$\{13,14,15\}$};
				
				\node[class] (E) at (-1.45,2.1)
				{$\{12\}$};
				
				\node[class] (D) at (2.5,-2.1)
				{$\{10,11\}$};
				
				\draw[cover] (A) -- (C);
				\draw[cover] (C) -- (B);
				\draw[cover] (C) -- (F);
				\draw[cover] (B) -- (E);
				
				\node[
				anchor=west,
				inner sep=0pt,
				font=\footnotesize
				] at (5.3,0) {%
					\begingroup
					\setlength{\arraycolsep}{1.4pt}%
					\renewcommand{\arraystretch}{.78}%
					$\begin{pmatrix}
						* & * & * & * & * & * & * & * & *
						& \gzero & \gzero & * & * & * & *
						\\
						* & * & * & * & * & * & * & * & *
						& \gzero & \gzero & * & * & * & *
						\\
						* & * & * & * & * & * & * & * & *
						& \gzero & \gzero & * & * & * & *
						\\
						\gzero & \gzero & \gzero & * & *
						& \gzero & \gzero & \gzero & \gzero
						& \gzero & \gzero & *
						& \gzero & \gzero & \gzero
						\\
						\gzero & \gzero & \gzero & * & *
						& \gzero & \gzero & \gzero & \gzero
						& \gzero & \gzero & *
						& \gzero & \gzero & \gzero
						\\
						\gzero & \gzero & \gzero & * & *
						& * & * & * & *
						& \gzero & \gzero & * & * & * & *
						\\
						\gzero & \gzero & \gzero & * & *
						& * & * & * & *
						& \gzero & \gzero & * & * & * & *
						\\
						\gzero & \gzero & \gzero & * & *
						& * & * & * & *
						& \gzero & \gzero & * & * & * & *
						\\
						\gzero & \gzero & \gzero & * & *
						& * & * & * & *
						& \gzero & \gzero & * & * & * & *
						\\
						\gzero & \gzero & \gzero & \gzero & \gzero
						& \gzero & \gzero & \gzero & \gzero
						& * & *
						& \gzero & \gzero & \gzero & \gzero
						\\
						\gzero & \gzero & \gzero & \gzero & \gzero
						& \gzero & \gzero & \gzero & \gzero
						& * & *
						& \gzero & \gzero & \gzero & \gzero
						\\
						\gzero & \gzero & \gzero & \gzero & \gzero
						& \gzero & \gzero & \gzero & \gzero
						& \gzero & \gzero & *
						& \gzero & \gzero & \gzero
						\\
						\gzero & \gzero & \gzero & \gzero & \gzero
						& \gzero & \gzero & \gzero & \gzero
						& \gzero & \gzero & \gzero
						& * & * & *
						\\
						\gzero & \gzero & \gzero & \gzero & \gzero
						& \gzero & \gzero & \gzero & \gzero
						& \gzero & \gzero & \gzero
						& * & * & *
						\\
						\gzero & \gzero & \gzero & \gzero & \gzero
						& \gzero & \gzero & \gzero & \gzero
						& \gzero & \gzero & \gzero
						& * & * & *
					\end{pmatrix}$%
					\endgroup
				};				
			\end{tikzpicture}
			\caption{\small Example of a preorder and its corresponding structural matrix ring.}\label{fig:PreorderExample}
		\end{figure}
	\end{exa}
	
	In contrast to the full matrix ring $\MnR$ (for $n\geq 2$), structural matrix rings need not be additively generated by their units. Nevertheless, the next lemma shows that the only possible obstruction comes from elements having a special shape: diagonal matrices with entries in positions corresponding to singleton equivalence classes.
	\begin{lem}\label{lem:SumOfDiagonalAndUnits}
		Let $A = (a_{ij}) \in \MprecR$. Then $A$ can be written as $A=A'+D$, where $A'$ is a sum of units of $\MprecR$, and $D$ is a diagonal matrix
		whose only nonzero entries are $a_{hh}$ in positions $(h,h)$, for those $h$ whose equivalence class $[h]$ is a singleton.
	\end{lem}
	\begin{proof}
		Let $I_n$ denote the $n\times n$ identity matrix, and $E_{ij}$ the $n \times n$ matrix with $1$ at position $(i,j)$ and zeros elsewhere.
		We show that $a_{ij}E_{ij}$ is a sum of units of $\MprecR$ for all $i,j$ satisfying $i\precsim j$ and either $i\neq j$ or, if $i=j$, then $[i]$ is not a singleton. The lemma then follows immediately from $A=\sum_{i\precsim j } a_{ij} E_{ij}$.
		
		If $i \precsim j$ but $i\neq j$, then $I_n + a_{ij}E_{ij}$ is a unit of $\MprecR$, and hence $a_{ij}E_{ij} = \left(I_n+ a_{ij}E_{ij}\right)-I_n$ is a sum of units.
		
		It remains to show that $a_{ii}E_{ii}$ is a sum of units whenever $[i]$ is not a singleton.
		As we have relabeled rows and columns so that the equivalence classes under $\sim$ are intervals, either $i-1 \in [i]$ or $i+1 \in [i]$.
		
		If $i-1 \in [i]$, the matrices
		\[
		U=
		\diag\left(
		-I_{i-2},
		\begin{pmatrix}
			0&-1\\
			-1&-1
		\end{pmatrix},
		-I_{n-i}
		\right),
		\]
		and
		
		\[
		V=
		\diag\left(
		I_{i-2},
		\begin{pmatrix}
			0&1\\
			1&a_{ii}+1
		\end{pmatrix},
		I_{n-i}
		\right)
		\]
		are units of $\MprecR$ and satisfy $U+V = a_{ii} E_{ii}$.
		
		If $i+1\in[i]$, the same argument applies with
		\[
		U=
		\diag\left(
		-I_{i-1},
		\begin{pmatrix}
			-1&-1\\
			-1&0
		\end{pmatrix},
		-I_{n-i-1}
		\right)
		\]
		and
		\[
		V=
		\diag\left(
		I_{i-1},
		\begin{pmatrix}
			a_{ii}+1&1\\
			1&0
		\end{pmatrix},
		I_{n-i-1}
		\right).
		\]
		Again, $U$ and $V$ are units and $U+V=a_{ii}E_{ii}$. In both cases, blocks of size zero are omitted.
	\end{proof}

	\section{Polynomial functions over structural matrix rings}
	In this section we establish some notation and prove three lemmas concerning the evaluation of polynomials in $\MprecR[x]$.
	
	The ring isomorphism $\MnR[x] \cong \M_n(R[x])$ clearly restricts to an isomorphism $\MprecR[x] \cong \M_{\precsim}( R[x])$, and we will switch freely between viewing an element of $\MprecR[x]$ either as a polynomial with matrix coefficients, or as a matrix whose entries are polynomials with scalar coefficients.
	The latter viewpoint is useful because evaluation at a fixed element is multiplicative for polynomials with coefficients in the commutative ring $R$.
	
	The following notation makes the isomorphism explicit.
	\begin{notn}\label{notn:polynomial}
		Let $f = F_0 + F_1x + \cdots + F_m x^m \in \MprecR[x]$, where $F_k = (f_{ij}^{(k)})$ for $k=0,\ldots,m$.
		Viewed as an element of $\M_\precsim (R[x])$, the polynomial $f$ becomes the matrix $f = \left(f_{ij}\right)$, where each $f_{ij} = \sum_{k=0}^m f_{ij}^{(k)} x^k$ is a polynomial in $R[x]$.
	\end{notn}
	\begin{defn}
		For $f \in \MprecR[x]$ as above and $A \in \MprecR$, we denote by $f(A)$ the result of substituting $A$ for the indeterminate $x$ to the right of the coefficients. In other words, $f(A) = \sum_{k=0}^m F_kA^k$.
	\end{defn}
	One could also define substitution to the left side of the coefficients, but we stick to right substitution throughout. The corresponding results for left substitution still hold and their proofs work analogously.
	
	\begin{rem}\label{rem:ProductEvaluation}
		Because the coefficients of the polynomials under consideration lie in a noncommutative ring, given two polynomials $f,g$ in $\MprecR[x]$ and $A \in \MprecR$, in general $(fg)(A) \neq f(A)g(A)$.
		Writing $f = \sum_i F_i x^i$ and $g = \sum_{j} G_j x^j$, we instead obtain
		\[
		(fg)(A) = \sum_{i,j}F_iG_jA^{i+j} =
		\sum_{i} F_i\left(\sum_{j}G_j A^j \right)A^i = (f\cdot (g(A)))(A),
		\]
		where $f \cdot (g(A))$ is the polynomial arising as the product of $f$ and the constant polynomial $g(A)$ in $\MprecR[x]$, which can then further be evaluated at $A$ to yield $(f\cdot (g(A)))(A)$.
	\end{rem}
	
	\begin{notn}\label{notn:MatrixRestriction}
			For a matrix $A \in \MnR$ and $i,j \in \{1,\ldots,n\}$, we denote the $(i,j)$-th entry of $A$ by $[A]_{ij}$. 
			We also define $A^{[i,j]}$, the restriction of $A$ to the set $[i,j] = \{h \in \{1,\ldots,n\} \mid i \precsim h \precsim j\}$, to be the matrix obtained from $A$ by replacing every entry whose row or column index is not in the set $[i,j]$ with zero.
	\end{notn}

	\begin{lem}\label{lem:PowerEqualsPowerOfRestriction}
		Let $A \in \MprecR$, and $i,j \in \{1,\ldots,n\}$. For all positive integers $k$, the $(i,j)$-th entry of $A^k$ only depends on $A^{[i,j]}$. That is, $[A^k]_{ij} = [\left(A^{[i,j]}\right)^k]_{ij}$.
	\end{lem}
	\begin{proof}
		Induction on $k$. The statement is clear for $k=1$. We assume now that the statement holds for some fixed $k$ and all $i,j$, and show that it also holds for $k+1$.
		Then
		\begin{align*}
			\left[A^{k+1} \right]_{ij}
			&= \left[A^k \cdot A \right]_{ij}
			=\sum_{h=1}^{n} \left[A^k\right]_{ih} \left[A\right]_{hj}
			= \sum_{h \in [i,j]} \left[A^k\right]_{ih} \left[A\right]_{hj},
		\end{align*}
		because for $h \notin [i,j]$, either $\left[A^k\right]_{ih}$ or $\left[A\right]_{hj}$ equals zero by the definition of $\MprecR$.
		Applying the induction hypothesis to $\left[A^k\right]_{ih}$, and noting the following two identities 
		\[
		\left(A^{[i,j]}\right)^{[i,h]} = A^{[i,h]}, \qquad \left(A^{[i,j]}\right)^{[h,j]} = A^{[h,j]},
		\]
	 	for all $h \in [i,j]$, we  obtain
		\begin{align*}
			\left[A^{k+1} \right]_{ij} &=\sum_{h \in [i,j]}\left[\left(A^{[i,h]}\right)^k\right]_{ih}\left[A^{[h,j]}\right]_{hj}\\
			&= \sum_{h \in [i,j]}\left[\left(A^{[i,j]}\right)^k\right]_{ih}\left[A^{[i,j]}\right]_{hj}\\
			&=  [\left(A^{[i,j]}\right)^{k+1}]_{ij}.
		\end{align*}
	\end{proof}
	\begin{lem}\label{lem:f(A)_ijIsSumOfScalarEvaluations}
		Let $f \in \MprecR[x]$, and let $A \in \MprecR$. Using Notation~\ref{notn:polynomial}, for all $i,j \in \{1,\ldots,n\}$
		\[
		\left[f(A)\right]_{ij} = \sum_{h \in [i,j]} \left[f_{ih}(A)\right]_{hj} = \sum_{h \in [i,j]} \left[f_{ih}(A^{[h,j]})\right]_{hj}.
		\]
	\end{lem}
	\begin{proof}
		\begin{align*}
			\left[f(A)\right]_{ij}&= \left[\sum_{k=0}^m F_k A^k  \right]_{ij} = \sum_{k=0}^m\left[ F_k A^k  \right]_{ij} = \sum_{k=0}^m\sum_{h =1}^n f^{(k)}_{ih} [A^k]_{hj}.
		\end{align*}
		Since $F_k$ and $A^k$ are elements of $\MprecR$, we have $f_{ih}^{(k)} = 0$ unless $i \precsim h$, while $[A^k]_{hj} =0 $ unless $h \precsim j$. Hence only the summands with $h \in [i,j]$ remain. By reversing the order of summation we obtain 
		\[
		\left[f(A)\right]_{ij} =
		\sum_{h\in [i,j]}\sum_{k=0}^m f^{(k)}_{ih} [A^k]_{hj} = \sum_{h\in [i,j]} \left[f_{ih}(A)\right]_{hj}.
		\]
		This shows the first equation of the lemma. The second one follows from Lemma~\ref{lem:PowerEqualsPowerOfRestriction}.
	\end{proof}
	
	\begin{lem}\label{lem:IthRowEquivalentSubsummands}
		Let $f \in \MprecR[x]$ be as in Notation~\ref{notn:polynomial} and $i \in \{1,\ldots,n\}$.
		Let $S$ be a subring of $R$ and $I$ an ideal of $S$. The following are equivalent:
		\begin{enumerate}
			\item\label{ithrowcondition} For every $A \in \MprecS$, all entries in the $i$-th row of $f(A)$ are elements of $I$.
			\item\label{subsummandscondition} For every $A \in \MprecS$, and for all $h,j \in \{1,\ldots,n\}$ satisfying $i \precsim h \precsim j$, 
			\[
			\sum_{h' \in [h]} \left[f_{ih'}(A)\right]_{h'j} \in I.
			\]
		\end{enumerate}
	\end{lem}
	\begin{proof}
		For the direction ($\ref{subsummandscondition}$) $\Rightarrow$ ($\ref{ithrowcondition}$), use Lemma~\ref{lem:f(A)_ijIsSumOfScalarEvaluations} to write the $(i,j)$-th entry of $f(A)$ as 
		\[
			\left[f(A)\right]_{ij} = \sum_{h \in [i,j]} \left[f_{ih}(A)\right]_{hj}.
		\]
		After partitioning the set $[i,j]$ into its $\sim$-equivalence classes, the contribution of each class $[h]$ is
		\[
		\sum_{h'\in[h]}[f_{ih'}(A)]_{h'j},
		\]
		which lies in $I$ by ($\ref{subsummandscondition}$). Hence $[f(A)]_{ij}$, which is the sum of these contributions, is also an element of the ideal $I$.
		
		For the other direction, assume (\ref{ithrowcondition}) and fix $j$. We use strong downward induction on the partially ordered set of equivalence classes contained in  $[i,j]$: Assuming the statement holds for all $h' \in [i,j]$ satisfying $h \precsim h'$ but $h' \nprecsim h$, we prove the statement for $h$. Note that for $h \sim j$ there are no such $h'$, so the induction hypothesis is vacuous.
		
		Let $A \in \MprecS$. Because the sum
		\[
		\sum_{h' \in [h]} \left[f_{ih'}(A)\right]_{h'j}
		\]
		only depends on $A^{[h,j]}$ by Lemma~\ref{lem:f(A)_ijIsSumOfScalarEvaluations}, we can assume $A=A^{[h,j]}$. By ($\ref{ithrowcondition}$), we have $[f(A)]_{ij} \in I$, and this element can be written as
		\begin{align*}
			[f(A)]_{ij} &=  \sum_{h' \in [i,j]} \left[f_{ih'}(A)\right]_{h'j}\\
			&=\sum_{\substack{h' \in [i,j]\\ h \nprecsim h'}}\left[f_{ih'}(A)\right]_{h'j} + \sum_{\substack{h' \in [h]}}\left[f_{ih'}(A)\right]_{h'j} + \sum_{\substack{h' \in [i,j]\\ h \precsim h'\\h' \nprecsim h}}\left[f_{ih'}(A)\right]_{h'j}.
		\end{align*}
		
		For the first sum, note that $h' \in [i,j]$ and $h \nprecsim h'$ imply $h' \notin [h,j]$. For such $h'$, all positive powers of $A$ have zero $h'$-th row, as $A=A^{[h,j]}$ by assumption.  Furthermore, $h' \notin [h,j]$ implies in particular that $h' \neq j$, so there is also no contribution from the constant term of $f_{ih'}$ to $\left[f_{ih'}(A)\right]_{h'j}$. Hence the first sum vanishes.
		
		The second sum is the element we want to prove lies in $I$, so if we show that the third sum is in $I$, we are done. 
		
		Note that in the case of $h \sim j$, the third sum is zero and hence trivially an element of $I$.
		For the other cases, let $\{h_1,\ldots,h_s\}$ be a complete set of representatives of the equivalence classes of elements $h' \in [i,j]$ satisfying $h \precsim h'$ but $h' \nprecsim h$. The induction hypothesis gives
		\[
		\sum_{\substack{h' \in [h_l]}}\left[f_{ih'}(A)\right]_{h'j} \in I,
		\]
		for each $l=1,\ldots,s$.
		Thus we conclude
		\[
		\sum_{\substack{h' \in [h]}}\left[f_{ih'}(A)\right]_{h'j} = [f(A)]_{ij} - \sum_{l=1}^s \sum_{\substack{h' \in [h_l]}}\left[f_{ih'}(A)\right]_{h'j} \in I.
		\]
	\end{proof}
	
	\section{Integer-valued polynomials and null polynomials}
	We now come to the main object of our study: integer-valued polynomials and null polynomials. 
	Throughout this section, let $S \subseteq R$ be commutative rings with the same identity, and let $I$ be an ideal of $S$.
	
	The following definition allows us to discuss both integer-valued and null polynomials at the same time.
	\begin{defn}\label{def:GeneralizedIVP}
		We define the set of generalized integer-valued polynomials as
		\[
		\Int_{\MprecR}\left(\MprecS,\MprecI\right)= \left\{f \in \MprecR[x] \mid \forall A\in\MprecS:\ f(A)\in\MprecI\right\}.
		\]
		Here $\MprecI$ denotes the set of matrices in $\MprecS$ whose
		entries in every permitted position belong to $I$.
	\end{defn}
	We are primarily interested in the following two special cases.
	\begin{defn}
		Let $D$ be an integral domain with quotient field $K$. For $R=K$, and $S=I=D$, Definition~\ref{def:GeneralizedIVP} specializes to
		\begin{align*}
			\Int_{\MprecK}\left(\MprecD\right) &=\Int_{\MprecK}\left(\MprecD,\MprecD\right)\\
			&= \left\{f \in \MprecK[x] \mid \forall A\in\MprecD:\ f(A)\in\MprecD\right\},
		\end{align*}
		the set of integer-valued polynomials over $\MprecD$ with coefficients in $\MprecK$.
	\end{defn}
	\begin{defn}
		Let $R$ be a commutative ring. Taking $S=R$ and $I= (0)$ instead gives
		\begin{align*}
			\N_{\MprecR}(\MprecR) &= \Int_{\MprecR}(\MprecR,0)\\
			&=\{f \in \MprecR[x] \mid \forall A\in\MprecR:\ f(A)=0\},
		\end{align*}
		the set of null polynomials with coefficients in $\MprecR$.
	\end{defn}
	
	We will show that $\Int_{\MprecK}\left(\MprecD\right)$ is a ring, that $\N_{\MprecR}(\MprecR)$ is a two-sided ideal of $\MprecR[x]$, and we will give an explicit description of their elements.
	
	We first record several elementary observations that will be used
	below; see also~\cite{Werner_IVPoverMatrixRings,Frisch_Ringsets}.
	
	\begin{rem}\label{rem:IVPAreLeftModule}
		$\Int_{\MprecR}(\MprecS,\MprecI)$ is a left $\MprecS$-module. Indeed, for two polynomials $f,g \in \Int_{\MprecR}(\MprecS,\MprecI)$, it is easy to see that their sum is again a generalized integer-valued polynomial. Furthermore, because $I$ is an ideal of $S$, for all $A,C \in \MprecS$, we immediately see that
		\[
		(Cf)(A) = C\cdot (f(A)) \in \MprecI,
		\]
		as $f(A) \in \MprecI$, and $C \in \MprecS$.
		
		The question whether $\Int_{\MprecR}(\MprecS,\MprecI)$ is
		also a right $\MprecS$-module, however, is less easy to answer.
	\end{rem}
	\begin{rem}\cite[cf. Lemma 2.3]{Werner_PolynomialsThatKill}\label{rem:NullPolyIdealIffClosedUnderRightMult}
		Recall from Remark~\ref{rem:ProductEvaluation} that for two polynomials $f$ and $g \in \MprecR[x]$, and all $A \in \MprecR$, even though possibly $(fg)(A) \neq f(A)g(A)$, the following identity holds.
		\[
		(fg)(A) = (f\cdot (g(A)))(A).
		\]
		Here, $f \cdot (g(A))$ is the polynomial arising as the product of $f$ and the constant polynomial $g(A)$ in $\MprecR[x]$, which can then further be evaluated at $A$ to yield $(f\cdot (g(A)))(A)$.
		
		It follows immediately that if $g$ is an element of $\N_{\MprecR}(\MprecR)$, and $f$ is an arbitrary polynomial, the product $fg$ is again a null polynomial. Hence $\N_{\MprecR}(\MprecR)$ is a left ideal of $\MprecR[x]$.
		
		Furthermore, the above identity shows that $\N_{\MprecR}(\MprecR)$ is a right ideal, and hence a two-sided ideal, of $\MprecR[x]$ if and only if $\N_{\MprecR}(\MprecR)$ is closed under right multiplication by constants from $\MprecR$. It also shows that the set $\Int_{\MprecK}\left(\MprecD\right)$ is closed under multiplication, and hence forms a ring, if and only if $\Int_{\MprecK}\left(\MprecD\right)$ is closed under right multiplication by constants from $\MprecD$. In particular, we will obtain both statements after proving that the set of generalized integer-valued polynomials is an $\MprecS$-$\MprecS$-bimodule in Theorem \ref{thm:IVPFormBimodule}.
	\end{rem}
	
	\begin{rem}\cite[cf. Theorem 2.7]{Frisch_Ringsets}
		The relation between integer-valued and null polynomials can be made explicit in the following way. Let $f \in \MprecK[x]$. Then $f$ can be written as $f_1/{c}$ for some $c \in D\setminus\{0\}$ and $f_1 \in \MprecD[x]$.
		Denote by $\overline{f_1}$ the image of $f_1$ in $\Mprec(D/cD)[x]$.
		Then
		\[
		f \in \Int_{\MprecK}\left(\MprecD\right) \Longleftrightarrow \overline{f_1} \in \N_{\Mprec(D/cD)}(\Mprec(D/cD)).
		\]
	\end{rem}
	\begin{rem}\label{rem:ConnectionNullPolsIVP}\cite[cf. Theorem 2.7]{Frisch_Ringsets}
		In view of the preceding remark, it is not surprising that the set of integer-valued polynomials $\Int_{\MprecK}\left(\MprecD\right)$ is closed under multiplication, and hence forms a ring, if and only if the null polynomials $\N_{\Mprec(D/cD)}(\Mprec(D/cD))$ form a two-sided ideal of $\Mprec(D/cD)[x]$ for every $c \in D\setminus \{0\}$.
	\end{rem}
	
	\begin{rem}\label{rem:ClosedUnderSumsOfUnits}\cite[cf. Theorem 1.2]{Werner_IVPoverMatrixRings}
		$\Int_{\MprecR}(\MprecS,\MprecI)$ is closed under multiplication from the right by units of $\MprecS$. Indeed, for $U$ a unit of $\MprecS$ and $f = \sum_{k} F_k x^k \in \Int_{\MprecR}(\MprecS,\MprecI)$, observe that for all $A \in \MprecS$,
		\begin{align*}
			(f\cdot U)(A) &= \sum_{k=0}^{m} F_k U A^k = \left(\sum_{k=0}^{m} F_k (UAU^{-1})^k\right)U \\
			&= f(UAU^{-1})U.
		\end{align*}
		As $f$ was assumed to be generalized integer-valued, $f(UAU^{-1})$ is an element of $\MprecI$, and since $I$ is an ideal of $S$, also $f(UAU^{-1})U \in \MprecI$ as desired.
		
		Because $\Int_{\MprecR}(\MprecS,\MprecI)$ is also closed under addition, $fC'$ is generalized integer-valued for all $C' \in \MprecS$ that can be written as a sum of units.
	\end{rem}
	Although, in general, not every element of $\MprecS$ is a sum of units, the set of generalized integer-valued polynomials nevertheless turns out to be a right $\MprecS$-module, which is the content of the following theorem.
	
	\begin{thm}\label{thm:IVPFormBimodule}
		$\Int_{\MprecR}(\MprecS,\MprecI)$ is a $\MprecS$-$\MprecS$-bimodule.
	\end{thm}
	\begin{proof}
		$\Int_{\MprecR}(\MprecS,\MprecI)$ is a left $\MprecS$-module by Remark~\ref{rem:IVPAreLeftModule}. Noting the associativity of matrix multiplication, it thus suffices to show that for every generalized integer-valued polynomial $f \in \Int_{\MprecR}(\MprecS,\MprecI)$ and all matrices $C \in \MprecS$, their product $f \cdot C$ is again an element of $\Int_{\MprecR}(\MprecS,\MprecI)$.
		
		By Lemma~\ref{lem:SumOfDiagonalAndUnits}, the matrix $C$ can be written as $C=C'+D$ where $C'$ is a sum of units, and $D$ is a diagonal matrix with nonzero entries only in positions $(h,h)$ for those $h$ whose equivalence class $[h]$ is a singleton.
		
		By Remark~\ref{rem:ClosedUnderSumsOfUnits}, the polynomial $fC'$ is in $\Int_{\MprecR}(\MprecS,\MprecI)$.
		As the set of generalized integer-valued polynomials is closed under addition, it only remains to show that $fD \in \Int_{\MprecR}(\MprecS,\MprecI)$.
		Using the closure under addition again, we reduce to the case where $D=\diag(0,\ldots,c_{kk},\ldots,0)$, for some $k$ satisfying $[k] = \{k\}$ and $c_{kk} \in S$.
		
		We show row by row that, for every $A \in \MprecS$, all entries of $(fD)(A)$ lie in $I$. Let $i \in \{1,\ldots,n\}$ be a row index.
		By Lemma~\ref{lem:IthRowEquivalentSubsummands}, the condition that all entries in the $i$-th row of $(fD)(A)$ are elements of $I$ is equivalent to the condition that for all $h,j \in \{1,\ldots,n\}$ with $i \precsim h \precsim j$ we have
		\[
		\sum_{h' \in [h]} \left[(fD)_{ih'}(A)\right]_{h'j} \in I,
		\]
		where $(fD)_{ih'} \in R[x]$ is the $(i,h')$-th entry of $fD$, when the latter is viewed as a matrix under the isomorphism $\MprecR[x] \cong \Mprec(R[x])$.
		
		Because of the way $D$ was defined, $(fD)_{ih'} = 0$, unless $h' = k$.
		Now if $h \neq k$,  then $k \notin [h]$, as the equivalence class of $k$ was assumed to be a singleton. Hence the above sum vanishes for all $h$ not equal to $k$. In these cases the sum is thus trivially an element of $I$.
		
		We are left with the case $h=k$. Since $f \in \Int_{\MprecR}(\MprecS,\MprecI)$, we know
		\[
		\sum_{h' \in [k]} \left[f_{ih'}(A)\right]_{h'j} = \left[f_{ik}(A)\right]_{kj} \in I.
		\]
		Because the coefficients of $f_{ik}$ are elements of $R$, they commute with $c_{kk}$.
		Therefore,
		\begin{align*}
			\sum_{h' \in [k]} \left[(fD)_{ih'}(A)\right]_{h'j} &= \left[(fD)_{ik}(A)\right]_{kj} = \left[(f_{ik}c_{kk})(A)\right]_{kj} = \left[(c_{kk}f_{ik})(A)\right]_{kj} \\
			&= c_{kk}\left[f_{ik}(A)\right]_{kj} \in I.
		\end{align*}
	\end{proof}
	The bimodule property has two consequences: first, it yields the ring and ideal properties we mentioned in Remark~\ref{rem:NullPolyIdealIffClosedUnderRightMult}; second, it allows us to isolate individual entries, which are scalar-coefficient polynomials. This will give an entrywise characterization of generalized integer-valued polynomials.
	
	\begin{defn}
		For a matrix $A \in \MprecR$ and an index $h \in \{1,\ldots,n\}$, we extend Notation~\ref{notn:MatrixRestriction} and define  $A^{[h,\infty)}$ to be the matrix obtained from $A$ by replacing every entry with row or column index outside the set $[h,\infty)=\{k \in \{1,\ldots,n\} \mid h \precsim k\}$ with zero. We define $\MprecR^{[h,\infty)}$ to be
		\[
		\MprecR^{[h,\infty)} = \{A^{[h,\infty)} \mid A \in \MprecR\}.
		\]
	\end{defn}
	\begin{rem}
		Note that $\MprecR^{[h,\infty)}$ is not a unital subring of $\MprecR$, since its identity $I_n^{[h,\infty)}$ generally differs from $I_n$. It is, however, a corner algebra of $\MprecR$.
		When evaluating polynomials in $R[x]$ at an element $A \in \MprecR^{[h,\infty)}$, we use the convention $A^0 := I_n^{[h,\infty)}$. Note that evaluation in the ambient ring $\MprecR$ and in the corner coincide at positions whose row and column indices are in $[h,\infty)$, so this convention does not introduce ambiguities in Lemma~\ref{lem:f(A)_ijIsSumOfScalarEvaluations}. 
	\end{rem}
		
	\begin{defn}
		For $h\in \{1,\ldots,n\}$ let
		\[
		\Int_R\left(
		\MprecS^{[h,\infty)},
		\MprecI^{[h,\infty)}
		\right)
		\]
		denote the set of all $f\in R[x]$ such that
		\[
		f(A)\in\MprecI^{[h,\infty)},
		\qquad
		\text{for every }A\in\MprecS^{[h,\infty)}.
		\]
		For the special case $R=K$, $S=I=D$, we write
		\[
			\Int_K\left(\MprecD^{[h,\infty)}\right) =\left\{f \in K[x] \mid \forall A\in\MprecD^{[h,\infty)}:\ f(A) \in \MprecD^{[h,\infty)} \right\},
		\]
		and for the special case $R=S$ and $I=(0)$, we write
		\[
		\N_{R}\left(\MprecR^{[h,\infty)}\right) = \left\{f \in R[x] \mid \forall A\in\MprecR^{[h,\infty)}:\ f(A)=0 \right\}.
		\]
	\end{defn}
	
	The following characterization of generalized integer-valued polynomials with coefficients from an arbitrary structural matrix ring in terms of polynomials with scalar coefficients is a generalization of~\cite[Theorem 4.2]{Frisch_PolFunOnUpperTriangularMatrixAlgebras} and~\cite[Theorem 7.2]{Frisch_IVPonAlgebras}.
	\begin{thm}\label{thm:CharacterizationOfIVPViaScalarPolynomials}
		For $f \in \MprecR[x]$ the following are equivalent:
		\begin{enumerate}
			\item \label{fIsIVP} $f \in \Int_{\MprecR}\left(\MprecS,\MprecI\right)$.
			\item\label{EveryfihIsScalarIVP} For all $i,h \in \{1,\ldots,n\}$ with $i \precsim h$,
			\[
			f_{ih} \in \Int_{R}\left(\MprecS^{[h,\infty)},\MprecI^{[h,\infty)}\right).
			\]
		\end{enumerate}
	\end{thm}
	\begin{proof}
		The implication ($\ref{EveryfihIsScalarIVP}$) $\Rightarrow$ ($\ref{fIsIVP}$) follows directly from Lemma~\ref{lem:f(A)_ijIsSumOfScalarEvaluations}.
		
		For the converse, let $f \in \Int_{\MprecR}\left(\MprecS,\MprecI\right)$, and fix $i,h \in \{1,\ldots,n\}$ with $i \precsim h$. Our goal is to show that for all $A \in \MprecS^{[h,\infty)}$
		\[
		f_{ih}(A) \in \MprecI^{[h,\infty)}.
		\]
		It is clear that $f_{ih}(A)$ is again a restricted matrix in $\MprecR^{[h,\infty)}$. Therefore, we only need to show that for all indices $k,l$ satisfying $h\precsim k \precsim l$, the $(k,l)$-th entry of $f_{ih}(A)$ lies in $I$.
		
		Define $g:=E_{ii}fE_{hk}$ (where $E_{ij}$ denotes the standard matrix unit).
		Since $E_{ii}$ and $E_{hk}$ are elements of $\MprecS$, Theorem~\ref{thm:IVPFormBimodule} gives $g \in \Int_{\MprecR}\left(\MprecS,\MprecI\right)$. In particular, all entries of $g(A)$ are in $I$. Calculating $[g(A)]_{il}$ and using Lemma~\ref{lem:f(A)_ijIsSumOfScalarEvaluations} we get
		\[
		[g(A)]_{il} = \sum_{s \in [i,l]} \left[g_{is}(A)\right]_{sl} = [f_{ih}(A)]_{kl} \in I.
		\]
	\end{proof}
	
	We can now apply our results to the case of integer-valued polynomials over $\MprecD$ for an integral domain $D$, as well as to null polynomials over $\MprecR$ for an arbitrary commutative ring $R$.
	\begin{cor}
		Let $D$ be an integral domain with quotient field $K$. Then the set $\Int_{\MprecK}(\MprecD)$ of integer-valued polynomials with coefficients in $
		\MprecK$ forms a subring of $\MprecK[x]$.
		
		Moreover, when its elements are viewed as matrices whose entries are polynomials with scalar coefficients, $\Int_{\MprecK}(\MprecD)$ can be described as
		\begin{align*}
			\Int_{\MprecK}(\MprecD)=
			\left\{
			(f_{ih})\in\mathrm M_{\precsim}(K[x])
			\;\middle|
			\begin{array}{l}
				f_{ih}\in
				\Int_K\left(
				\MprecD^{[h,\infty)}
				\right)
				\\
				\text{for every $i,h$ with $i\precsim h$}
			\end{array}
			\right\}.
		\end{align*}
	\end{cor}
	\begin{proof}
		By Remark~\ref{rem:NullPolyIdealIffClosedUnderRightMult} and Theorem~\ref{thm:IVPFormBimodule}, $\Int_{\MprecK}(\MprecD)$ is a subring of $\MprecK[x]$.
		
		The characterization in terms of scalar-coefficient polynomials follows from Theorem~\ref{thm:CharacterizationOfIVPViaScalarPolynomials}.
	\end{proof}
	
	Although Werner's conjecture that null polynomials over a finite ring form a two-sided ideal~\cite[Conjecture 3.3]{Werner_PolynomialsThatKill} is false in general~\cite{Havlovec_NullPolynomials}, the statement is true for all structural matrix rings.
	\begin{cor}
		Let $R$ be a commutative ring. Then the set $\N_{\MprecR}(\MprecR)$ of null polynomials with coefficients in $\MprecR$ is a two-sided ideal of $\MprecR[x]$.
		
		Moreover, when its elements are viewed as matrices whose entries are polynomials with scalar coefficients, $\N_{\MprecR}(\MprecR)$ can be described as
		\[
		\N_{\MprecR}(\MprecR) =
		\left\{
		(f_{ih})\in\mathrm M_{\precsim}(R[x])
		\;\middle|
		\begin{array}{l}
			f_{ih}\in
			\N_R\left(
			\MprecR^{[h,\infty)}
			\right)
			\\
			\text{for every $i,h$ with $i\precsim h$}
		\end{array}
		\right\}.
		\]
	\end{cor}
	\begin{proof}
		Remark~\ref{rem:NullPolyIdealIffClosedUnderRightMult} together with Theorem~\ref{thm:IVPFormBimodule} yields that $\N_{\MprecR}\left(\MprecR\right)$ is a two-sided ideal of $\MprecR[x]$.
		
		The characterization in terms of scalar-coefficient polynomials follows from Theorem~\ref{thm:CharacterizationOfIVPViaScalarPolynomials}.
	\end{proof}	
\bibliographystyle{amsplain-doi}
\bibliography{bibliography}

@article{SedighiHafshejaniNaghipourRismanchian_IVPoverBlockMatrixAlgebras,
	author = {{Sedighi Hafshejani}, J. and Naghipour, A. R. and Rismanchian, M. R.},
	doi = {10.1142/S021949882050053X},
	journal =  {J. Algebra Appl.},
	number = {3},
	pages = {2050053},
	note = {17 pp.},
	title = {Integer-valued polynomials over block matrix algebras},
	volume = {19},
	year = {2020}
}

@article{SedighiHafshejani_IVPSubsetsMatrices,
author = {{Sedighi Hafshejani}, J. and Naghipour,  A. R. and Sakzad, A.},
title = {Integer-valued polynomials over subsets of matrix rings},
journal =  {Comm. Algebra},
volume = {47},
number = {3},
pages = {1077--1090},
year = {2019},
doi = {10.1080/00927872.2018.1499926},
}

@book{Chabert_IVP2025,
	author = {Chabert, Jean-Luc},
	series = {Colloquium Publications},
	address = {Providence, Rhode Island},
	publisher = {American Mathematical Society},
	title = {Integer-valued polynomials: from combinatorics to number theory, p-adic analysis, commutative and non-commutative algebra},
	doi = {10.1090/coll/069},
	year = {2025},
	volume    = {69}
}

@article{PeruginelliWerner_DecompIVP,
	author = {Peruginelli, Giulio and Werner, Nicholas J.},
	doi = {10.1016/j.jpaa.2017.10.007},
	journal =  {J. Pure Appl. Algebra},
	number = {9},
	pages = {2562--2579},
	title = {Decomposition of integer-valued polynomial algebras},
	volume = {222},
	year = {2018}
}

@unpublished{Havlovec_NullPolynomials,
		author = {Havlovec, Valentin},
		title = {Null polynomials over a finite ring need not form a two-sided ideal},
		note = {Preprint},
		year = {2026},
		eprint = {2608.13251}
	}

@article{Werner_IVPoverMatrixRings,
	author = {Werner, Nicholas J.},
	doi = {10.1080/00927872.2011.606859},
	journal =  {Comm. Algebra},
	number = {12},
	pages = {4717--4726},
	title = {Integer-valued polynomials over matrix rings},
	volume = {40},
	year = {2012}
}

@article{SmithVanWyk_StructuralMatrixRings,
  author  = {Smith, Kirby C. and van Wyk, Leon},
  title   = {An Internal Characterisation of Structural Matrix Rings},
  journal =  {Comm. Algebra},
  volume  = {22},
  number  = {14},
  pages   = {5599--5622},
  year    = {1994},
  doi     = {10.1080/00927879408825149}
}

@incollection{Werner_IVPAlgebrasSurvey,
  author    = {Werner, Nicholas J.},
  title     = {Integer-Valued Polynomials on Algebras: A Survey of
               Recent Results and Open Questions},
  booktitle = {Rings, Polynomials, and Modules},
  editor    = {Fontana, Marco and Frisch, Sophie and Glaz, Sarah and
               Tartarone, Francesca and Zanardo, Paolo},
  pages     = {353--375},
  publisher = {Springer},
  address  = {Cham},
  year      = {2017},
  doi       = {10.1007/978-3-319-65874-2_18}
}

@article{Werner_IVPSubsetsQuaternions,
title = {Integer-valued polynomials on subsets of quaternion algebras},
journal =  {J. Algebra},
volume = {686},
pages = {195--219},
year = {2026},
doi = {10.1016/j.jalgebra.2025.08.011},
author = {Werner, Nicholas J. }
}

@article{Werner_PolynomialsThatKill,
	author = {Werner, Nicholas J.},
	doi = {10.1142/S0219498813501119},
	journal =  {J. Algebra Appl.},
	pages = {1350111},
	note = {12 pp.},
	number = {3},
	title = {Polynomials that kill each element of a finite ring},
	volume = {13},
	year = {2014}
}

@article{Werner_IVPQuaternions,
	author = {Werner, Nicholas J.},
	doi = {10.1016/j.jalgebra.2010.06.024},
	journal =  {J. Algebra},
	number = {7},
	pages = {1754--1769},
	title = {Integer-valued polynomials over quaternion rings},
	volume = {324},
	year = {2010}
}

@article{Frisch_IVPonAlgebras,
	author = {Frisch, Sophie},
	doi = {10.1016/j.jalgebra.2012.10.003},
	journal =  {J. Algebra},
	pages = {414--425},
	title = {Integer-valued polynomials on algebras},
	volume = {373},
	year = {2013}
}

@article{Frisch_Ringsets,
	author = {Frisch, Sophie},
	title = {Polynomial functions on subsets of non-commutative rings — a link between ringsets and null-ideal sets},
	doi = {10.1051/itmconf/20182001003},
	journal =  {ITM Web Conf.},
	year = {2018},
	pages ={01003},
	volume = {20},
}

@article{Frisch_PolFunOnUpperTriangularMatrixAlgebras,
	author = {Frisch, Sophie},
	doi = {10.1007/s00605-016-1013-y},
	journal =  {Monatsh. Math.},
	number = {2},
	pages = {201--215},
	title = {Polynomial functions on upper triangular matrix algebras},
	volume = {184},
	year = {2017}
}
\end{document}